\documentclass[12pt,reqno]{article}
\usepackage{authblk}
\usepackage{amsmath,amssymb,amsfonts,amscd,latexsym,amsthm,mathrsfs}
\newtheorem{theorem}{Theorem}

\theoremstyle{remark}

\let\wt\widetilde
\renewcommand{\d}{{\mathrm d}}
\newcommand{\id}{\operatorname{id}}

\newcommand{\ord}{\operatorname{ord}}

\begin{document}

\title{Euler's factorial series and global relations}





\author[1]{Tapani Matala-aho}
\author[2,3]{Wadim Zudilin}

\affil[1]{Matematiikka, PL 3000, 90014 Oulun yliopisto, Finland}
\affil[2]{Institute for Mathematics, Astrophysics and Particle Physics, Radboud~Universiteit, PO Box 9010, 6500 GL Nijmegen, The Netherlands}
\affil[3]{School of Mathematical and Physical Sciences, The~University~of~Newcastle, Callaghan, NSW 2308, Australia}

\date{7 March 2017}

\maketitle

\rightline{\it To the memory of Marc Huttner \rm(1947--2015)}

\begin{abstract}
Using Pad\'e approximations to the series $E(z)=\sum_{k=0}^\infty k!\,(-z)^k$,
we address arithmetic and analytical questions related to its values in both $p$-adic and Archimedean valuations.
\end{abstract}

\section{Introduction}

In his 1760 paper on divergent series \cite{Eu60}, L.~Euler introduced and studied the formal (hypergeometric) series
\begin{equation}\label{Eulerseries}
E(z):=\sum_{k=0}^\infty k!\,(-z)^k
\end{equation}
(see also \cite{BL76,Ba79} and, in particular, \cite[Section 2.5]{La13}),
which specializes at $z=1$ to Wallis' series
\begin{equation}\label{Wallisseries}
W:=\sum_{k=0}^\infty(-1)^kk!\,,
\end{equation}
which has teased people's imagination since the times of Euler.
Notice that the series \eqref{Eulerseries} is a perfectly convergent $p$-adic series in the disc $|z|_p\le1$ for all primes $p$,
so that one can discuss the arithmetic properties of its values, for example, at $z=1$ (the Wallis case) and at $z=-1$.
The irrationality of the latter specialization,
\begin{equation}\label{Kurepaseries}
K=K_p:=\sum_{k=0}^\infty k!\,,
\end{equation}
for any prime $p$ is a folklore conjecture~\cite[p.~17]{Sch84}
(see~\cite{MS04}, also for a link of the problem to a combinatorial conjecture of H.~Wilf).
Already the expectation $|K_p|_p=1$ (so that $K_p$ is a $p$-adic unit) for all primes $p$,
which is an equivalent form of Kurepa's conjecture \cite{Ku71} from 1971, remains open;
see \cite{Ch15} for the latest achievements in this direction.

In what follows, we refer to the series in \eqref{Eulerseries} as Euler's factorial series and label it $E_p(z)$ when treat it
as the function in the $p$-adic domain for a given prime~$p$.

\begin{theorem}\label{EulerGlobal1}
Given $\xi\in\mathbb Z\setminus\{0\}$, let $\mathcal P$ be a subset of prime numbers such that
\begin{equation}
\label{cond}
\limsup_{n\to\infty} c^nn!\prod_{p\in\mathcal P}|n!|_p^2=0,
\quad\text{where}\; c=c(\xi;\mathcal P):=4|\xi|\prod_{p\in\mathcal P}|\xi|_p^2.
\end{equation}
Then either there exists a prime $p\in\mathcal P$ for which $E_p(\xi)$~is irrational, or there are two distinct primes $p,q\in\mathcal P$
such that $E_p(\xi)\ne E_q(\xi)$ \textup(while $E_p(\xi),E_q(\xi)\in \mathbb{Q}$\textup).
\end{theorem}

Because $\prod_p|n!|_p=1/n!$ when the product is taken over all primes, condition \eqref{cond} is clearly satisfied for
any subset $\mathcal P$ whose complement in the set of all primes is finite. This also suggests more exotic choices of $\mathcal P$.
Furthermore, the conclusion of the theorem is contrasted with Euler's sum
\begin{equation*}
\sum_{k=0}^\infty k\cdot k!=-1,
\end{equation*}
which is valid in any $p$-adic valuation (and follows from $\sum_{k=0}^{n-1}k\cdot k!=n!-1$)\,---\,an example of what is called a \emph{global} relation.

The idea behind the proof of Theorem \ref{EulerGlobal1}
is to construct approximations $p_n/q_n$ to the number $\omega_p=E_p(\xi)$ in question, which do not depend on $p$ and approximate the number considerably well
for each $p$-adic valuation: $q_n\omega_p-p_n=r_{n,p}$ for $n=0,1,2,\dots$; $|r_{n,p}|_p\to0$ as $n\to\infty$ and there are infinitely many indices $n$
for which $r_{n,p}\ne0$ for at least one prime $p\in\mathcal P$.
Assume that $\omega_p=a/b$, the same rational number, for all $p\in\mathcal P$.
Then $q_na-p_nb\in\mathbb Z\setminus\{0\}$, so that $0<|q_na-p_nb|_p\le1$, for infinitely many indices $n$ and \emph{all} primes $p$, hence
\begin{align*}
1&=|q_na-p_nb|\prod_p|q_na-p_nb|_p
\le|q_na-p_nb|\prod_{p\in\mathcal P}|q_na-p_nb|_p
\displaybreak[2]\\
&\le(|a|+|b|)\max\{|q_n|,|p_n|\}\prod_{p\in\mathcal P}|br_{n,p}|_p
\\
&\le(|a|+|b|)\max\{|q_n|,|p_n|\}\prod_{p\in\mathcal P}|r_{n,p}|_p
\end{align*}
for those $n$. This means that the condition
\begin{equation}
\label{cond-gen}
\limsup_{n\to\infty}\max\{|q_n|,|p_n|\}\prod_{p\in\mathcal P}|r_{n,p}|_p=0
\end{equation}
contradicts the latter estimate, thus making the $\mathcal P$-global linear relation $\omega_p=a/b$ impossible.

The result in Theorem \ref{EulerGlobal1} can be put in a general context of global relations for Euler-type series;
the corresponding settings can be found in the paper \cite{BCY04}. We do not pursue this route here as our principal motivation
is a sufficiently elementary arithmetic treatment of an analytical function that have some historical value.
The rational approximations to $E(z)$ we construct in Section~\ref{sec-Pade} are Pad\'e approximations; in spite of being known for centuries,
implicitly from the continued fraction for $E(z)$ given by Euler himself \cite{Eu60} and explicitly from the work of T.\,J.~Stieltjes \cite{St94},
these Pad\'e approximations remain a useful source for arithmetic and analytical investigations.
In Section~\ref{sec-divergent} we revisit Euler's summation of Wallis' series \eqref{Wallisseries} using the approximations;
we also complement the derivation by providing `Archimedean analogue(s)' for the divergent series $K=E(-1)$ from \eqref{Kurepaseries}\,---\,the case
when the classical strategy does not work.

Our way of constructing the Pad\'e approximations is inspired by a related Pad\'e construction of M.~Hata and M.~Huttner in~\cite{HH02}.
The construction was a particular favourite of Marc Huttner who, for the span of his mathematical life, remained a passionate advocate
of interplay between Picard--Fuchs linear differential equations and Pad\'e approximations. We dedicate this work to his memory.

\section{Hypergeometric series and Pad\'e approximations}
\label{sec-Pade}

Euler's factorial series \eqref{Eulerseries} is the particular $a=1$ instance of the hypergeometric series
\begin{equation}\label{hyperseries}
{}_2F_0(a,1\mid z)=\sum_{k=0}^\infty(a)_kz^k.
\end{equation}
Here and in what follows, we use the Pochhammer notation $(a)_k$ which is defined inductively by
$(a)_0=1$ and $(a)_{k+1}=(a+k)(a)_k$ for $k\in\mathbb{Z}_{\ge 0}$. Our Pad\'e approximations below are given more generally
for the function~\eqref{hyperseries}.

\begin{theorem}\label{th-Pade}
For $n,\lambda\in\mathbb{Z}_{\ge 0}$, take
\begin{equation*}
B_{n,\lambda}(z)=\sum_{i=0}^n \binom ni\frac{(-1)^iz^{n-i}}{(a)_{i+\lambda}}.
\end{equation*}
Then $\deg_zB_{n,\lambda}=n$ and for a polynomial $A_{n,\lambda}(z)$ of degree $\deg_zA_{n,\lambda}\le n+\lambda-1$ we have
\begin{equation}\label{Pade}
B_{n,\lambda}(z)\,{}_2F_0(a,1\mid z)-A_{n,\lambda}(z)=L_{n,\lambda}(z),
\end{equation}
where
$\ord_{z=0}L_{n,\lambda}(z)=2n+\lambda$.
Explicitly,
\begin{equation}\label{remainderclosed}
L_{n,\lambda}(z)= (-1)^n n!\, z^{2n+\lambda} \sum_{k=0}^{\infty} k! \binom{n+k}k \binom{n+k+a+\lambda-1}k z^k.
\end{equation}
\end{theorem}

\begin{proof}
Relation \eqref{Pade} means that there is a `gap' of length $n$ in the power series expansion
\begin{equation*}
B_{n,\lambda}(z)\, {}_2F_0(a,1\mid z)=A_{n,\lambda}(z)+L_{n,\lambda}(z).
\end{equation*}
Write
$B_{n,\lambda}(z)= \sum_{h=0}^{n} b_{h}t^{h} $
and consider the series expansion of the product
\begin{equation*}
B_{n,\lambda}(z)\,{}_2F_0(a,1\mid z)=\sum_{l=0}^{\infty}r_lz^l,
\end{equation*}
where $r_l=\sum_{h+k=l}b_{h}(a)_k$; in particular,
\begin{equation*}
r_l=\sum_{i=0}^n (-1)^i\binom ni\frac{(a)_{i+\lambda+m}}{(a)_{i+\lambda}}
=\sum_{i=0}^n (-1)^i\binom ni(a+i+\lambda)_m
\end{equation*}
with $m=l-n-\lambda$ for $l>n+\lambda-1$. To verify the desired `gap' condition,
\begin{equation}\label{gap-cond}
r_{n+\lambda}=r_{n+\lambda+1}=\dots=r_{n+\lambda+n-1}=0,
\end{equation}
we introduce the shift operators $N=N_a$ and $\Delta=\Delta_a=N-\id$ defined on functions $f(a)$ by
\begin{equation*}
Nf(a)=f(a+1) \quad\text{and}\quad \Delta f(a)=f(a+1)-f(a).
\end{equation*}
It follows from
\begin{equation*}
\Delta^n (a+\lambda)_m=(-1)^n(-m)_n (a+\lambda+n)_{m-n}
\quad\text{for}\; n\in\mathbb{Z}_{\ge 0}
\end{equation*}
that
\begin{align*}
r_l
&=\sum_{i=0}^n (-1)^i\binom ni N^i(a+\lambda)_m
=(\id-N)^n (a+\lambda)_m
\\
&=(-\Delta)^n (a+\lambda)_m
=(-m)_n(a+\lambda+n)_{m-n},
\end{align*}
which, in turn, implies \eqref{gap-cond} because of the vanishing of $(-m)_n$ for $m=0,1,\dots,n-1$.
The explicit expression for $r_l$ just found also gives
\begin{equation*}
r_{2n+\lambda+k}=(-n-k)_n (a+\lambda+n)_k
=(-1)^n n!\,\binom{n+k}k\,k!\,\binom{n+k+a+\lambda-1}k,
\end{equation*}
hence the closed form \eqref{remainderclosed}. We also have
\begin{equation*}
A_{n,\lambda}(z)=\sum_{l=0}^{n+\lambda-1}r_lz^l
=\sum_{l=0}^{n+\lambda-1}z^l\sum_{\substack{i=0\\i\ge n-l}}^n (-1)^i\binom ni\frac{(a)_{i+l-n}}{(a)_{i+\lambda}}.
\end{equation*}
This concludes our proof of the theorem.
\end{proof}

From now on we choose $a=1$, $\lambda=0$, change $z$ to $-z$ and renormalize the corresponding Pad\'e approximations
produced by Theorem~\ref{th-Pade} by  multiplying them by $(-1)^nn!$\,:
\begin{align*}
Q_n(z):=(-1)^nn!\,B_{n,0}(-z)
&=\sum_{i=0}^n\binom ni \frac{n!}{i!}\,z^{n-i}
=\sum_{i=0}^n i!\,{\binom ni}^2z^i,
\\
P_n(z):=(-1)^nn!\,A_{n,0}(-z)
&=(-1)^n\sum_{l=0}^{n-1}(-z)^l\sum_{i=n-l}^n (-1)^{n-i}\binom ni\frac{n!\,(i+l-n)!}{i!}
\\
&=(-1)^n\sum_{l=0}^{n-1}(-z)^{l}\sum_{i=0}^{l} (-1)^ii!\,(l-i)!\,{\binom ni}^2
\\ \intertext{and}
R_n(z):=(-1)^nn!\,L_{n,0}(-z)
&=n!^2z^{2n}\sum_{k=0}^\infty(-1)^kk!\,{\binom{n+k}k}^2z^k.
\end{align*}
In this notation, the Pad\'e approximation formula \eqref{Pade} may be rewritten as
\begin{equation}\label{PADE3}
Q_n(z)E(z)-P_n(z)=R_n(z)
\end{equation}
for $n\in\mathbb{Z}_{>0}$. Observe that
\begin{equation}\label{detqp}
Q_n(z) P_{n+1}(z) - Q_{n+1}(z) P_n(z)=n!^2 z^{2n}.
\end{equation}
Indeed, the standard Pad\'e walkabout proves the identity
\begin{equation*}
Q_n(z) P_{n+1}(z) - Q_{n+1}(z) P_n(z)=Q_{n+1}(z) R_n(z) - Q_n(z) R_{n+1}(z),
\end{equation*}
in which the degree of the left-hand side is at most $2n$ while the order of the right-hand side is at least $2n$.

\begin{proof}[Proof of Theorem \textup{\ref{EulerGlobal1}}]
We may use the formal series identity \eqref{PADE3} to get appropriate numerical approximations for any
at $z=\xi\in\mathbb Z\setminus\{0\}$. For $n=1,2,\dots$, take $p_n=P_n(\xi)\in\mathbb Z$, $q_n=Q_n(\xi)\in\mathbb Z$ and define
$$
r_{n,p}=R_n(\xi)=q_nE_p(\xi)-p_n
$$
for each prime $p\in\mathcal P$. Using elementary summation formulas and trivial estimates for binomials we have
\begin{align*}
|q_n|&\le |\xi|^nn!\sum_{i=0}^n{\binom ni}^2=|\xi|^nn!\binom{2n}n<4^n|\xi|^nn!\,,
\\
|p_n|&\le|\xi|^{n-1}n\sum_{i=0}^{n-1} i!\,(n-1-i)!\,{\binom ni}^2
\le|\xi|^{n}n!\sum_{i=0}^n{\binom ni}^2<4^n|\xi|^nn!
\end{align*}
and
\begin{equation*}
|r_{n,p}|_p=|\xi|_p^{2n}|n!|_p^2\biggl|\sum_{k=0}^\infty(-1)^kk!\,{\binom{n+k}k}^2\xi^k\biggr|_p \le|\xi|_p^{2n}|n!|_p^2.
\end{equation*}
Therefore, condition \eqref{cond-gen} reads
$$
\limsup_{n\to\infty}4^n|\xi|^nn!\prod_{p\in\mathcal P}|\xi|_p^{2n}|n!|_p^2=0
$$
and because for at least one $p\in\mathcal P$ we have either $r_{n,p}\ne0$ or $r_{n+1,p}\ne0$ from \eqref{detqp},
the theorem follows.
\end{proof}

\section{Summation of divergent series}
\label{sec-divergent}

Fix a prime $p$.
If we restrict, for simplicity, the sequence of indices $n$ to the arithmetic progression $n\equiv0\bmod p$, then
it is not hard to see from the calculation in Section~\ref{sec-Pade} that the sequence of rational approximations $p_n/q_n=P_n(\xi)/Q_n(\xi)$
converges $p$-adically to $E(\xi)$. Indeed,
\begin{equation*}
q_n=1+\sum_{i=1}^{p-1}i!\,{\binom ni}^2\xi^i+\sum_{i=p}^n i!\,{\binom ni}^2\xi^i
\equiv1\bmod p
\end{equation*}
is a $p$-adic unit (all the binomials are divisible by $p$ in the first sum, while $i!$ is divisible by $p$ in the second one), and
\begin{equation*}
\biggl|E(\xi)-\frac{p_n}{q_n}\biggr|_p
=|q_n|_p^{-1}|q_nE(\xi)-p_n|_p=|r_{n,p}|_p
\to0 \quad\text{as}\; n\to\infty.
\end{equation*}

When $\xi>0$, the same sequence of rational numbers $p_n/q_n=P_n(\xi)/Q_n(\xi)$ converges to the value at $z=\xi$ of the integral
\begin{equation}\label{tilde-E}
\tilde E(z)=\int_0^\infty\frac{e^{-s}}{1+zs}\,\d s
\end{equation}
with respect to the Archimedean norm. The integral itself converges for any $z\in\mathbb C\setminus(-\infty,0]$,
though its formal Taylor expansion at $z=0$ (obtained by expanding $1/(1+zs)$ into the power series under the integral sign)
is precisely Euler's factorial series $E(z)$ as in~\eqref{Eulerseries}. In particular, relation \eqref{PADE3} remains valid
with $E(z)$ replaced by $\tilde E(z)$ and $R_n(z)$ by
\begin{equation*}
\tilde R_n(z)=n!\,z^{2n}\int_0^\infty\frac{s^ne^{-s}}{(1+zs)^{n+1}}\,\d s
\quad\text{for}\; n=0,1,2,\dots,
\end{equation*}
so that
\begin{equation*}
\biggl|\tilde E(\xi)-\frac{p_n}{q_n}\biggr|
=\frac{\wt R_n(\xi)}{Q_n(\xi)}
\le\frac{n^n}{(n+1)^{n+1}}\to0 \quad\text{as}\; n\to\infty,
\end{equation*}
where the estimates $Q_n(\xi)\ge n!\,\xi^n$ and
\begin{equation*}
\tilde R_n(\xi)\le n!\,\xi^{2n}\biggl(\max_{s>0}\frac{s^n}{(1+\xi s)^{n+1}}\biggr)\int_0^\infty e^{-s}\,\d s
=n!\,\xi^n\frac{n^n}{(n+1)^{n+1}}
\end{equation*}
were used. This computation reveals us that
\begin{equation*}
\lim_{n\to\infty}\frac{p_n}{q_n}
=\tilde E(\xi)
=-xe^x\biggl(\gamma+\log x+\sum_{k=1}^\infty\frac{(-x)^k}{k\cdot k!}\biggr)\bigg|_{x=1/\xi},
\end{equation*}
where $\gamma=0.5772156649\ldots$ is Euler's constant.
In particular,
\begin{equation*}
W=\tilde E(1)
=e\biggl(-\gamma+\sum_{k=1}^\infty\frac{(-1)^{k-1}}{k\cdot k!}\biggr)
=0.5963473623\ldots
\end{equation*}
for Wallis' series \eqref{Wallisseries}. The resulted quantity is known as the Euler--Gompertz constant \cite{La13}.

The strategy in the previous paragraph does not apply to $z=\xi<0$, somewhat already observed by Stieltjes in \cite{St94}.
The analytical continuation of the the function $\tilde E(z)$ depends on whether we perform the integration along the upper or lower banks of the ray $[0,\infty)$
in~\eqref{tilde-E}; denote the corresponding values by $\tilde E_+(z)$ and $\tilde E_-(z)$, respectively.
By considering the integration of $e^{-s}/(1+zs)$ along the curvilinear triangle that consists of the segment $[0,R]$ (along a particular bank),
the arc $[R,R\,e^{\sqrt{-1}\theta}]$ followed by the segment $[R\,e^{\sqrt{-1}\theta},0]$, where $0<\theta<\pi/2$ for the upper bank and $-\pi/2<\theta<0$ for
the lower one, and then taking the limit as $R\to\infty$ (so that the integral along the arc tends to~$0$) we conclude that
\begin{equation}\label{xi<0}
\tilde E_\pm(z)=\int_0^{e^{\sqrt{-1}\theta}\infty}\frac{e^{-s}}{1+zs}\,\d s
=-xe^x\biggl(\gamma+\log|x|\pm\sqrt{-1}\pi+\sum_{k=1}^\infty\frac{(-x)^k}{k\cdot k!}\biggr)\bigg|_{x=1/z},
\end{equation}
with the choice of $\theta$ arbitrary in the interval $0<\theta<\pi/2$ for $\tilde E_+(z)$ and
in the interval $-\pi/2<\theta<0$ for $\tilde E_-(z)$.
In particular,
\begin{align*}
K=\tilde E_{\pm}(-1)
&=\frac1e\biggl(\gamma+\sum_{k=1}^\infty\frac1{k\cdot k!}\mp\sqrt{-1}\pi\biggr)
\\
&=0.6971748832\ldots\mp\sqrt{-1}\cdot1.1557273497\ldots
\end{align*}
for the series in \eqref{Kurepaseries}.

\section{Final remarks}
\label{sec-final}

In \cite{Zu12} we outline a different strategy of proving a result analogous to Theorem \ref{EulerGlobal1}
on using the Hankel determinants generated by the tails of Euler's factorial series \eqref{Eulerseries}.
As the condition on a subset of primes $\mathcal P$ in that result is spiritually similar to \eqref{cond},
we do not detail the derivation here. However we stress that a potential combination of the two methods,
namely, using the Hankel determinants generated by the Pad\'e approximations of Euler's factorial series,
may be a source of further novelties on the topic. A discussion on this type of construction in the Archimedean setting can be found in \cite{Zu17}.

One consequence of the formula in \eqref{xi<0}, which uncovers a pair of \emph{complex} conjugate values for $E(\xi)$ when $\xi<0$,
is that the \emph{rational} approximations $p_n/q_n=P_n(\xi)/Q_n(\xi)$ do not converge at all in such cases. Interestingly enough,
the Hankel determinants `see' those complex values \eqref{xi<0} as experimentally observed in~\cite{Zu12}.

Finally, we would like to note that nothing is known about the irrationality and transcendence of the Archimedean valuations of
Euler's factorial series \eqref{Eulerseries} at rational $z=\xi$ (see the discussion in \cite[Sections 3.15, 3.16]{La13}).
This is in contrast with its $q$-analogue
$$
\sum_{k=0}^\infty z^k\prod_{i=1}^k(1-q^i),
$$
for which irrationality and linear independence results are known in Archimedean and non-Archimedean places alike\,---\,see~\cite{Ma09}.
Further details on a nice $q$-counterpart of the Pad\'e approximation analysis can be found in~\cite{Ma11}.


\end{document}